\documentclass[12pt,letterpaper]{amsart}

\usepackage{amsmath,amssymb,amsthm}
\usepackage{fullpage}
\usepackage{verbatim}
\usepackage[noadjust]{cite} 
\usepackage{hyperref}
\usepackage[capitalize, nameinlink]{cleveref}
\usepackage{enumitem}
\usepackage[usenames,dvipsnames]{xcolor} 
\usepackage{tikz}
\usepackage[font=small,labelfont=bf, margin=.5in]{caption}

\colorlet{prettygreen}{ForestGreen!60!LimeGreen}

\makeatletter
\def\@defaultbiblabelstyle#1{[#1]}
\makeatother

\usetikzlibrary{calc, arrows.meta}
\tikzset{vtx/.style={circle, fill, inner sep=1.5pt}}
\tikzset{openvtx/.style={circle, draw, inner sep=1.5pt}}

\usepackage{tikz-cd}

\newtheorem{theorem}{Theorem}[section]
\newtheorem{lemma}[theorem]{Lemma}
\newtheorem{proposition}[theorem]{Proposition}
\newtheorem{corollary}[theorem]{Corollary}

\newtheorem*{claim*}{Claim}

\theoremstyle{definition}
\newtheorem{definition}[theorem]{Definition}

\newtheorem{question}[theorem]{Question}
\newtheorem{problem}[theorem]{Problem}

\theoremstyle{remark}
\newtheorem{remark}[theorem]{Remark}

\newenvironment{prop}{\begin{proposition}}{\end{proposition}}

\crefname{claim}{Claim}{Claims}

\newlist{homtenum}{enumerate}{1}
\setlist[homtenum,1]{leftmargin=36pt}

\newcommand{\N}{{\mathcal N}}

\newcommand{\Z}{{\mathbb Z}}

\newcommand{\pieven}{\pi_1^{\mathrm{EVEN}}}

\DeclareMathOperator{\id}{id}

\DeclareMathOperator{\Cay}{Cay}
\DeclareMathOperator{\wind}{wind}

\DeclareMathOperator{\squareop}{\square}

\title{Abelian groups without $3$-chromatic Cayley graphs}
\author{Mike Krebs}
\address{Department of Mathematics, California State University --- Los
Angeles}
\email{mkrebs@calstatela.edu}
\author{Maya Sankar}
\thanks{Sankar was supported by a Fannie and John Hertz Foundation Fellowship and NSF Graduate Research Fellowship DGE-1656518.}
\address{Department of Mathematics, Stanford University}
\email{mayars@stanford.edu}


\makeatletter
\let\mytitle\@title
\let\myauthor\@author
\makeatother

\begin{document}

\begin{abstract}Let $G$ be an abelian group.  The main theorem of this paper asserts that there exists a Cayley graph on $G$ with chromatic number $3$ if and only if $G$ is not of exponent $1$, $2$, or $4$. For connected Cayley graphs, we also show that this theorem holds when $G$ is finitely generated.  
Although motivated by ideas from algebraic topology, our proof may be expressed purely combinatorially.
As a by-product, we derive a topological result which is of independent interest. Suppose $X$ is a connected non-bipartite graph, and let $\N(X)$ denote its neighborhood complex. We show that if the fundamental group $\pi_1(\N(X))$ or first homology group $H_1(\N(X))$ is torsion, then the chromatic number of $X$ is at least $4$. This strengthens a special case of a classical result of Lov\'asz, which derives the same conclusion if $\pi_1(\N(X))$ is trivial.  
\end{abstract}


\maketitle

\section{Introduction}

Let $G$ be a group and $S$ a \emph{symmetric} subset of $G$---that is, $x\in S$ if and only if $x^{-1}\in S$. The \emph{Cayley graph} $\Cay(G,S)$ of $G$ with generating set $S$ is defined to be the graph on vertex set $G$ in which two vertices $x$ and $y$ are adjacent if and only if $x^{-1}y\in S$. 

Chromatic numbers of Cayley graphs have been extensively investigated \cite{Alon,Cze}, motivated in part by applications to information theory \cite{AlOr,Rix}.
Special focus \cite{Kim, Kok, Lau, Ngo, Zie} has been given to the case of \emph{binary Cayley graphs}, also called \emph{cube-like graphs}, which are Cayley graphs over the group $G=(\Z/2\Z)^m$. Here, a somewhat surprising result of Payan \cite{Pay} shows that binary Cayley graphs cannot have chromatic number 3.

\begin{theorem}[Payan \cite{Pay}]\label{thm:Payan}
If $G=(\Z/2\Z)^m$ and $S$ is a symmetric subset of $G$, then $\chi(\Cay(G,S))\neq 3$.
\end{theorem}

Several alternate proofs of \cref{thm:Payan} appear in the literature \cite{Bea, Cer}, and a proof of a special case was given by Sokolov\'a \cite{Sok} five years prior. Both \cite{Pay,Bea} rely on the fact---shown combinatorially by Payan, but proven seven years prior by Stiebitz \cite{Sti} using topological methods---that a generalized Mycielskian of an odd cycle always has chromatic number 4. The use of the generalized Mycielskian construction suggests that \cref{thm:Payan} might be a fundamentally topological result, as topological methods have proven key in studying the chromatic number of generalized Mycielskians more generally \cite{GyJeSt,Mul}. In \cite{Cer}, Cervantes and the first author wondered whether this topological connection could be made more explicit. 

This paper provides a complete characterization of the abelian groups $G$ over which there exists a 3-chromatic Cayley graph $\Cay(G,S)$. 
Our proof studies the topological structure of the \emph{entire} Cayley graph---thus answering the question from \cite{Cer} and deriving a fundamentally new proof of \cref{thm:Payan}. In contrast, \cite{Bea,Cer,Pay} each prove \cref{thm:Payan} by locating a subgraph of the binary Cayley graph which was previously known to not be 3-colorable.
Our characterization relies on the following quantity. 
Given a group $G$, its \emph{exponent} $\exp(G)$ is defined to be the least positive integer $k$ such that $g^k=\id$ for all $g\in G$; we set $\exp(G)=0$ if no such $k$ exists. 

\begin{theorem}\label{thm:chi3iff}
    Let $G$ be an abelian group. There exists a symmetric subset $S$ of $G$ such that $\chi(\Cay(G,S))=3$ if and only if $\exp(G)\notin\{1,2,4\}$.
\end{theorem}

Our methods also prove a partial generalization of a groundbreaking 1978 result (\cref{thm:lovasz}) of Lov\'asz \cite{Lov}, which he used to prove the Kneser conjecture. He defined the \emph{neighborhood complex} of a graph $X$ to be a simplicial complex $\N(X)$ with vertex set $V(X)$, whose faces are exactly those sets $W\subset V(X)$ which have a common neighbor. (To avoid confusion, we always use the variable $X$ for graphs, reserving the variable $G$ for groups.) In one of the first results in topological combinatorics, he showed that the chromatic number of a graph can be lower-bounded in terms of the homotopical connectivity of its neighborhood complex. A topological space $Y$ is defined to be \emph{$k$-connected} if $Y$ is path-connected and its first $k$ homotopy groups $\pi_1(Y),\ldots,\pi_k(Y)$ are trivial---or equivalently, if $Y$ is path-connected and for $1\leq d\leq k$, every continuous map $\mathbb S^d\to Y$ extends continuously to a map on the disk $\mathbb D^d$.

\begin{theorem}[Lov\'asz \cite{Lov}]\label{thm:lovasz}
Let $X$ be a graph. If $\N(X)$ is $k$-connected then $\chi(X)\geq k+3$.	
\end{theorem}

Lov\'asz asked whether \cref{thm:lovasz} could be strengthened by considering homology instead of homotopy. We provide an affirmative answer to this question in the case $k=1$. We show that it suffices for the first homology group $H_1(\N(X))$ to be \emph{torsion}, meaning that every element of $H_1(\N(X))$ has finite order. As we remark below, this is a strictly weaker condition than the fundamental group $\pi_1(\N(X))$ being torsion, which is in turn weaker than $\N(X)$ being 1-connected. 
To our knowledge, this is the first result that weakens the condition in \cref{thm:lovasz} on the topological structure of $\N(G)$ for any $k$.

\begin{theorem}\label{thm:homology}
Let $X$ be a connected non-bipartite graph. If $H_1(\N(X))$ is torsion then $\chi(X)\geq 4$.	
\end{theorem}

To contextualize this result, let us note that $\N(X)$ is path-connected if and only if $X$ is connected and non-bipartite. Moreover, a standard result in algebraic topology \cite[Theorem 2A.1]{Hat} states that $H_1(\N(X))$ is the abelianization of $\pi_1(\N(X))$ in this case. Thus, $H_1(\N(X))$ is torsion or trivial whenever $\pi_1(\N(X))$ is torsion or trivial, respectively. It follows that \cref{thm:homology} is a strict strengthening of the $k=1$ case of \cref{thm:lovasz}.  Moreover, the following corollary of \cref{thm:homology} (which we prove separately in \cref{sec:obstructions}) is itself a strengthening of the $k=1$ case of \cref{thm:lovasz}.

\begin{corollary}\label{cor:pi1-torsion}
Let $X$ be a connected non-bipartite graph. If $\pi_1(\N(X))$ is torsion then $\chi(X)\geq 4$.	
\end{corollary}

In the literature there are several competing notions of homotopy equivalence, fundamental group, etc.\ for graphs ({\it cf}.\ \cite{Mal}). The article \cite{San} discusses one such set of concepts particularly useful for our purposes. (These ideas seem to have been developed independently elsewhere as well; the articles \cite{matsushita} and \cite{wrochna} also introduce identical or equivalent definitions.) In particular, \cite{San} defines a discrete fundamental group $\pi_1(X)$ of a graph $X$ which is particularly easy to calculate for Cayley graphs over abelian groups $G$. (This $\pi_1(X)$ is denoted $\pi^2_1(X)$ in \cite{matsushita}.) Additionally, $\pi_1(X)$ contains $\pi_1(\N(X))$ as a subgroup. Thus, one direction of \cref{thm:chi3iff} follows by showing the discrete fundamental group $\pi_1(\Cay(G,S))$ is torsion and applying \cref{cor:pi1-torsion}.

The remainder of this paper is organized as follows.
In \cref{sec:3-chromatic} we show that \cref{thm:chi3iff} does not hold for groups of any other exponent. 
In \cref{sec:discrete} we introduce the preliminaries needed to prove non-3-chromaticity in \cref{thm:chi3iff}. Although topologically motivated, the objects introduced in this section are defined purely combinatorially and require no topological background to appreciate.
In \cref{sec:obstructions}, we prove a theorem that implies both \cref{thm:chi3iff} and \cref{cor:pi1-torsion} as corollaries, then use the same ideas to prove \cref{thm:homology}. 
We conclude in \cref{sec:concluding} with some open questions and additional remarks.

\section{Groups of Exponent Not Dividing $4$}\label{sec:3-chromatic}

Call a graph \emph{$k$-chromatic} if it has chromatic number exactly $k$. We first show that any group of exponent other than $1$, $2$, or $4$ admits a 3-chromatic Cayley graph.

\begin{lemma}\label{lem:3-chromatic}
Let $G$ be a group with $\exp(G)\notin\{1,2,4\}$. Then there exists a symmetric subset $S$ of $G$ such that $\chi(\Cay(G,S))=3$.
\end{lemma}

\begin{proof}
If $G$ contains an element $x$ of infinite order, let $S=\{x^{\pm 1},x^{\pm 2}\}$. The graph $\Cay(\langle x\rangle,S)$ is isomorphic to $\Cay(\Z,\{\pm 1,\pm2\})$ which is non-bipartite---as $\{0,1,2\}$ induces a 3-cycle---but 3-colorable by reduction modulo 3. Then $\Cay(G,S)$ is a disjoint union of copies of $\Cay(\langle x\rangle,S)$, hence 3-chromatic.

If $G$ is torsion then some $y\in G$ has order $m$ not dividing 4, or else $\exp(G)$ would divide 4. Because $4\nmid m$, either $8\mid m$ or $(2k+1)\mid m$ for some $k\geq 1$. In the former case, let $x=y^{m/8}$, which has order 8, and let $S=\{x^{\pm 1},x^4\}$. Then $\Cay(\langle x\rangle,S)$ is isomorphic to the graph $\Cay(\Z/8\Z,\{\pm 1,4\})$ pictured in \Cref{fig:Z8} which is non-bipartite and 3-colorable. 
\begin{figure}[h]
\begin{tikzpicture}[
	a/.style={vtx,draw=black,fill=red},
	b/.style={vtx,draw=black,fill=blue},
	c/.style={vtx,draw=black,fill=prettygreen}]
\foreach \vtx [count=\i] in {a,b,c,a,b,c,a,b} {
	\coordinate[\vtx] (v\i) at (45*\i:0.9);
}
\draw (v1) -- (v2) -- (v3) -- (v4) -- (v5) -- (v6) -- (v7) -- (v8) -- (v1);
\draw (v1) -- (v5) (v2) -- (v6) (v3) -- (v7) (v4) -- (v8);
\end{tikzpicture}
\caption{A 3-coloring of the graph $\Cay(\Z/8\Z,\{\pm 1,4\})$}\label{fig:Z8}
\end{figure}
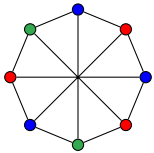
In the latter case, let $x=y^{m/(2k+1)}$ and $S=\{x^{\pm 1}\}$. Then $\Cay(\langle x\rangle,S)$ is isomorphic to $C_{2k+1}$ which is 3-chromatic. In either case, $\Cay(G,S)$ is a disjoint union of copies of the 3-chromatic graph $\Cay(\langle x\rangle,S)$.
\end{proof}

\begin{remark}
When $G$ is finitely generated, we can insist that the set $S$ in \cref{lem:3-chromatic} generates $G$, i.e., that the corresponding Cayley graph is connected --- see \cref{thm:connected-Cayley}.
\end{remark}

\begin{remark}
In the case that $G$ contains an element $y$ of even order $m\geq 6$, one could alternatively have taken $S=\{y^{\pm 1},y^{\pm 4}\}$. One can verify directly that $\Cay(\langle y\rangle,S)$ has chromatic number $3$; it also follows from a formula of Heuberger's for the chromatic number of a circulant graph with $2$ generators \cite{Heu}.
\end{remark}

\section{A discrete variant of the fundamental group}\label{sec:discrete}

In this section, we introduce the discrete fundamental group and the winding number, which are key to the proofs of our main results. Although these objects are topologically motivated, they are defined purely combinatorially, and require no topological prerequisites.

\subsection{The discrete fundamental group}\label{sec:discrete-pi1}
We now lay the foundation for our main results, beginning with a discretized analogue of the fundamental group for graphs, introduced in \cite{San}.

\begin{definition}\label{def:homotopy}
Let $X$ be a graph. A \emph{walk} in $X$ of \emph{length} $\ell$ is a sequence of vertices $v_0\cdots v_\ell$ such that $v_iv_{i+1}$ is an edge of $X$ for all $0\leq i<\ell$. We say two walks $W$ and $W'$ are \emph{homotopy equivalent} in $X$ if we may transform $W$ into $W'$ via a finite sequence of \emph{substitutions}, \emph{insertions}, and \emph{deletions}, defined as follows.
\begin{homtenum}
\item[(Sub)] Given a walk $v_0\cdots v_\ell$ and an index $0<i<\ell$, replace $v_i$ with $v'_i$ for some $v'_i\in N(v_{i-1})\cap N(v_{i+1})$. 
\item[(Ins)] Given a walk $v_0\cdots v_\ell$ and an index $0\leq i\leq \ell$, replace $v_i$ with $v_iwv_i$ for some $w\in N(v_i)$.
\item[(Del)] Given a walk $v_0\cdots v_\ell$ and an index $0<i<\ell$ satisfying $v_{i-1}=v_{i+1}$, replace $v_{i-1}v_iv_{i+1}$ with $v_{i-1}$.
\end{homtenum}
If a walk is homotopy equivalent to a length-$0$ walk, then we say it is \emph{null-homotopic}. 
\end{definition}

It is straightforward to verify that homotopy equivalence yields an equivalence relation on walks. If $X$ is a graph and $v_0$ is a vertex of $X$, we define the \emph{discrete fundamental group} $\pi_1(X,v_0)$ with \emph{base vertex} $v_0$ to be the set of all closed walks in $X$ beginning and ending at $v_0$, modulo homotopy equivalence. The group operation on $\pi_1(X,v_0)$ is given by $[W_1]*[W_2]:=[W_1*W_2]$, where $[W]$ denotes the homotopy equivalence class of a walk $W$, and $W_1*W_2$ denotes the concatenation of $W_1$ followed by $W_2$. 
Let $\pieven(X,v_0)$ be the subgroup of $\pi_1(X,v_0)$ consisting of all equivalence classes represented by walks of even length---this is well-defined because substitution, insertion, and deletion preserve the length of a walk modulo 2.  

For a connected graph $X$, the second author \cite{San} showed that $\pi_1(X,v_0)$ and $\pieven(X,v_0)$ are independent, up to group isomorphism, of the choice of $v_0$. Thus, we generally omit the base vertex $v_0$ and simply write $\pi_1(X)$ or $\pieven(X)$, much like the analogous notation for the fundamental group in algebraic topology. 
The discrete fundamental group also has connections to the neighborhood complex $\N(X)$, which we recall is the simplicial complex on $V(X)$ whose faces are subsets $W\subset V(X)$ with a common neighbor in $X$. The second author constructed an isomorphism between $\pieven(X)$ and $\pi_1(\N(X))$ when $\N(X)$ is connected---which occurs if and only if $X$ is connected and non-bipartite.
\begin{prop}[{\cite[Prop. 6.3]{San}}]\label{prop:pi1even}Let $X$ be a connected non-bipartite graph.  Then \[\pieven(X)\cong \pi_1(\mathcal{N}(X)).\]
\end{prop}

\subsection{The discrete fundamental group of an abelian Cayley graph} 
\label{sec:compute-pi1}
We provide a formula for the discrete fundamental group of a Cayley graph over an abelian group, as well as the fundamental group of its neighborhood complex. Let $G$ be an abelian group, written using additive notation, and let $S$ be a symmetric subset of $G$.

Let $\mathbb{Z}[S]$ be the free $\mathbb Z$-module on $S$. We write elements of $\mathbb{Z}[S]$ as formal linear combinations of elements of $S$ with coefficients in $\mathbb{Z}$. A typical element will be written in the form $a_1\cdot s_1+\cdots+ a_k\cdot s_k$, where $a_1,\dots,a_k\in\mathbb{Z}$ and $s_1,\dots,s_k\in S$.  Let $R[S]$ be the kernel of the natural homomorphism from $\mathbb{Z}[S]$ to $G$ given by $a_1\cdot s_1+\cdots+ a_k\cdot s_k\mapsto a_1 s_1+\cdots+ a_k s_k$.  That is, $R[S]$ is the \emph{group of relations} for $S$ in $G$. Let $H[S]$ be the subgroup of $R[S]$ generated by all elements of the form $1\cdot s+1\cdot (-s)$ for $s\in S$, as well as elements of the form $1\cdot s_1+1\cdot s_2-1\cdot s_3-1\cdot s_4$ for $s_1, s_2, s_3, s_4\in S$ such that $s_1+s_2=s_3+s_4$. The group $H[S]$ is the \emph{group of homotopy relations} for $S$.

\begin{proposition}\label{prop:pi1-for-abelian-Cayley-graph}
Suppose $G$ is an abelian group and $S\subset G$ a symmetric subset that generates $G$. Then $\pi_1(\Cay(G,S))\cong R[S]/H[S]$.
\end{proposition}

\begin{proof}
In this proof, it is convenient to denote a closed walk by the generators used to take each step rather than by the vertices visited. Given $s_1,\ldots,s_\ell\in S$, let $(v_0;s_1,\ldots, s_\ell)$ denote the walk $v_0\dots v_\ell$ of length $\ell$ with $v_i:=v_0+s_1+\cdots+s_i$, and let $[v_0;s_1,\ldots,s_\ell]$ denote the homotopy equivalence class of this walk. 
Observe that $(v_0;s_1,\dots, s_\ell)$ is homotopy equivalent to $(v_0;s_1,\dots,s_{i-1},s_{i+1},s_i,s_{i+2},\dots,s_\ell)$ via a single substitution. That is, permuting the generators $s_1,\ldots,s_\ell$ has no effect on the homotopy equivalence class of the walk $(v_0;s_1,\ldots,s_\ell)$.
Additionally, if $s_{i+1}=-s_i$ then $(v_0;s_1,\dots,s_\ell)$ is homotopy equivalent to $(v_0;s_1,\ldots,s_{i-1},s_{i+2},\dots,s_\ell)$. Because we may freely reorder the generators, it follows that homotopy equivalence is preserved by deleting any two generators satisfying $s_i=-s_j$.

Let $\phi:R[S]\to\pi_1(\Cay(G,S),0)$ be the map that takes $a_1\cdot s_1+\cdots+a_k\cdot s_k\in R[S]$ to $[0;\pm s_1,\ldots,\pm s_1,\ldots,\pm s_k,\ldots,\pm s_k]$, which features either $s_i$ repeated $a_i$ times (if $a_i>0$) or $-s_i$ repeated $-a_i$ times (if $a_i<0$). We claim that $\phi$ is a group homomorphism. If $R_1=a_1\cdot s_1+\cdots+a_k\cdot s_k$ and $R_2=b_1\cdot s_1+\cdots +b_k\cdot s_k$ are elements of $R[S]$ then $\phi(R_1)*\phi(R_2)$ takes the form
$$\phi(R_1)*\phi(R_2)=[0;\pm s_1,\ldots,\pm s_1,\pm s_k\ldots,\pm s_k,\pm s_1,\ldots,\pm s_1,\pm s_k\ldots,\pm s_k].$$
By reordering the generators and canceling pairs of the form $(s_i,-s_i)$, we may write
\[
\phi(R_1)*\phi(R_2)=[0;\pm s_1,\ldots,\pm s_1,\ldots,\pm s_k,\ldots,\pm s_k],
\]
featuring $|a_i+b_i|$ copies of either $s_i$ (if $a_i+b_i>0$) or $-s_i$ (if $a_i+b_i<0$). This is exactly $\phi(R_1+R_2)$.

Additionally, $\phi$ is surjective, as $[0;s_1,\ldots,s_\ell]$ is the image of $1\cdot s_1+\cdots+1\cdot s_\ell$. Thus, $\pi_1(\Cay(G,S),0)\cong R[S]/\ker\phi$. It holds that
\[
\ker\phi=\{(1\cdot s_1+\cdots+1\cdot s_\ell)-(1\cdot s'_1+\cdots+1\cdot s'_m) \mid [0;s_1,\ldots,s_\ell]=[0;s'_1,\ldots,s'_m]\}
\]
is generated by differences of the form $(1\cdot s_1+\cdots +1\cdot s_\ell)-(1\cdot s'_1+\cdots+1\cdot s'_m)$ where $W=(0;s_1,\ldots,s_\ell)$ and $W'=(0;s'_1,\ldots,s'_m)$ differ by a single insertion, deletion, or substitution. If $W$ and $W'$ differ by a single insertion or deletion, the resulting difference is of the form $\pm(1\cdot s+1\cdot (-s))$. If $W$ and $W'$ differ by a single substitution, the resulting difference is of the form $1\cdot s_i+1\cdot s_{i+1}-1\cdot s'_i-1\cdot s'_{i+1}$, where these four generators satisfy $s_i+s_{i+1}=s'_i+s'_{i+1}$.
These are exactly the generators of $H[S]$, so $\ker\phi=H[S]$ and $\pi_1(\Cay(G,S),0)\cong R[S]/\ker\phi\cong R[S]/H[S]$.
\end{proof}

\begin{remark}
Let $R^\mathrm{EVEN}[S]\subset R[S]$ be the subgroup of all relations $a_1\cdot s_1\cdots a_k\cdot s_k\in R$ such that $a_1+\cdots+a_k$ is even. The same logic implies that $\pieven(\N(\Cay(G,S)))\cong R^{\mathrm{EVEN}}[S]/H[S]$. By \cref{prop:pi1even}, this yields a formula for the fundamental group of the neighborhood complex.
\end{remark}

\subsection{Winding number}
We now introduce a discrete notion of winding number for closed walks, similar to that for loops in topological spaces.  Let $X$ be a graph, together with a proper 3-coloring $c:V(X)\to \Z/3\Z$, which we may regard as a graph homomorphism $X\to K_3=\Cay(\mathbb{Z}/3\mathbb{Z},\{\pm 1\})$. Informally, the quantity we are about to define counts the net number of times the image of a closed walk $W$ in $X$ winds around the triangle $\Cay(\mathbb{Z}/3\mathbb{Z},\{\pm 1\})$.  We remark that analogues of the winding number have been used to lower-bound chromatic numbers of specific graphs in several contexts (see e.g.\ \cite{Arc,GuNaTa, Moh, Pay,Sim}).

\begin{definition}\label{def:wind}
Let $X$ be a graph and let $c:V(X)\to\mathbb{Z}/3\mathbb{Z}$ be a proper 3-coloring of $X$---that is, $c$ induces a graph homomorphism $X\to\Cay(\mathbb{Z}/3\mathbb{Z},\{\pm 1\})$. Let $W=v_0\cdots v_{\ell-1}v_0$ be a closed walk in $X$, and set $v_\ell:=v_0$ for convenience. The \emph{discrete winding number} of $W$ with respect to $c$, denoted $\wind(W,c)$, is defined to be $1/3$ of the following quantity: the number $a$ of indices $i=0,\dots,{\ell-1}$ such that $c(v_{i+1})-c(v_i)=1$ minus the number $b$ of indices $i=0,\dots,{\ell-1}$ such that $c(v_{i+1})-c(v_i)=-1$.
\end{definition}

Observe that $\wind(W,c)$ is always an integer, as $a-b\equiv\sum_{i=0}^{\ell-1} c(v_i)-c(v_{i+1})= 0\pmod 3$ because $W$ is a closed walk. 
We now make a crucial observation that winding number is invariant under homotopy equivalence.

\begin{proposition}\label{prop:winding-homotopy}
Let $X$ be a graph, and let $c$ be a $3$-coloring of $X$, with notation as above. If $W_1$ and $W_2$ are homotopy-equivalent walks in $X$, then $\wind(W_1,c)=\wind(W_2,c)$.
\end{proposition}

\begin{proof}
It suffices to show that the discrete winding number is preserved under the substitution, insertion, and deletion steps described in \cref{def:homotopy}. Write $W_1=v_0\cdots v_{\ell-1}v_0$.

\textit{Substitution.} Suppose $W_2=v_0\cdots v_{i-1}v'_iv_{i+1}\cdots v_{\ell-1}v_0$ is obtained from $W_1$ by replacing $v_i$ with some common neighbor $v_i'$ of $v_{i-1}$ and $v_{i+1}$. If $c(v_i)=c(v'_i)$, then it is clear that $\wind(W_1,c)=\wind(W_2,c)$. If $c(v_i)\neq c(v'_i)$, then $c(v_{i-1})$ is uniquely determined from $c(v_i)$ and $c(v'_i)$, as $v_{i-1}$ cannot have the same color as either $v_i$ or $v'_i$. The color of $v_{i+1}$ is uniquely determined in an identical fashion, so it follows that $c(v_{i+1})=c(v_{i-1})$. In this case, $c(v_i)-c(v_{i-1})$ and $c(v_{i+1})- c(v_i)$ must be 1 and $-1$ in some order, and $c(v'_i)-c(v_{i-1})$ and $c(v_{i+1})-c(v'_i)$ must be $1$ and $-1$ in the opposite order. It follows that the values $a$ and $b$ in \cref{def:wind} are the same for both $W_1$ and $W_2$, and thus $\wind(W_1,c)=\wind(W_2,c)$.

\textit{Insertion.} Suppose $W_2=v_0\cdots v_iwv_i\cdots v_{\ell-1}v_0$ is obtained from $W_1$ by a single insertion of some $w\in N(v_i)$. Because $c$ is a proper coloring, we have $c(w)\neq c(v_i)$, and it follows that $c(v_i)-c(w)$ and $c(w)-c(v_i)$ are 1 and $-1$ in some order. Hence, the values $a$ and $b$ in \cref{def:wind} are incremented by 1 each by the insertion. Their difference remains unchanged, so $\wind(W_1,c)=\wind(W_2,c)$.

\textit{Deletion.} If $W_2$ is formed from $W_1$ by a single deletion then $W_1$ is formed from $W_2$ by an insertion. Thus, the argument from the previous case applies.
\end{proof}

\begin{remark}\label{remark:pushforward}
From \cref{prop:winding-homotopy}, it immediately follows that $\wind(-,c)$ gives a group homomorphism $\pi_1(X)\to\Z$. Indeed, if $W_1$ and $W_2$ are two closed walks with the same base vertex $v_0$, then it is clear that $\wind(W_1*W_2,c)=\wind(W_1,c)+\wind(W_2,c)$.

An alternate perspective on this result is as follows. If $c:X\to K_3$ is a graph homomorphism (representing a proper 3-coloring of $X$), then $c$ induces a natural group homomorphism $c_*:\pi_1(X)\to\pi_1(K_3)$. A straightforward computation shows $\pi_1(K_3)\cong\Z$ and indeed, one can check that $c_*$ and $\wind(-,c)$ are the same group homomorphism (up to automorphism of $\Z$).
\end{remark}

\section{Obstructions to $3$-colorability}\label{sec:obstructions}

The key result of this section is \cref{thm:torsion-implies-chi4}, showing that if the discrete fundamental group of a connected non-bipartite graph $X$ is torsion, then $\chi(X)\geq 4$. We derive \cref{thm:chi3iff} as an immediate corollary.
\cref{thm:torsion-implies-chi4} also implies \cref{cor:pi1-torsion} (which we recall was a special case of \cref{thm:homology}) directly. This section concludes with a proof of the full statement of \cref{thm:homology}, whcih requires a little more abstract algebra, motivated by the ideas in the proof of \cref{cor:pi1-torsion}.

\begin{theorem}\label{thm:torsion-implies-chi4}
    Let $X$ be a connected non-bipartite graph. If $\pi_1(X)$ is torsion, then $\chi(X)\geq 4$.
\end{theorem}\begin{proof}Assume to the contrary that $X$ admits a proper 3-coloring $c:V(X)\to\Z/3\Z$. 
Because $X$ is non-bipartite, there is an odd-length closed walk $W=v_0\cdots v_{2\ell}v_0$ in $X$. We claim that $\wind(W,c)$ is odd. Indeed, by \cref{def:wind}, it holds that $3\times\wind(W,c)$ is the difference $a-b$ of two integers satisfying $a+b=2\ell+1$. Thus,
$$\wind(W,c)\equiv 3\times\wind(W,c)=a-b\equiv a+b\equiv 1\pmod 2.$$
Let $W^k$ denote the $k$-fold concatenation of $W$ with itself. We see that $\wind(W^k,c)=k\times\wind(W,c)\neq 0$ for any positive integer $k$. However, because $\pi_1(X)$ is torsion, $W^k$ is null-homotopic for some $k>0$.
\cref{prop:winding-homotopy} then implies that $\wind(W^k,c)=\wind(v_0,c)=0$, a contradiction.
\end{proof}

\begin{proof}[Proof of \cref{thm:chi3iff}]
Let $G$ be an abelian group, written using additive notation. If $G$ is not of exponent $1$, $2$, or $4$ then it admits a 3-chromatic Cayley graph by \cref{lem:3-chromatic}.
Now, suppose $\exp(G)$ divides 4. Let $S$ be a symmetric subset of $G$, and let $G_0$ be the subgroup of $G$ generated by $S$. Then, $\Cay(G,S)$ is a disjoint union of copies of the connected subgraph $X_0:=\Cay(G_0,S)$, so $\chi(\Cay(G,S))=\chi(X_0)$. By \cref{thm:torsion-implies-chi4}, it suffices to show that $\pi_1(X_0)$ is torsion.

\cref{prop:pi1-for-abelian-Cayley-graph} shows that $\pi_1(X_0)\cong R[S]/H[S]$, where $R[S]$ and $H[S]$ are as defined in \cref{sec:compute-pi1}. We show that for all $q\in R[S]$, it holds that $4q\in H[S]$. For any $s\in S$, we have $s+s=(-s)+(-s)$, so $2\cdot s-2\cdot (-s)\in H[S]$. Then,
\[4\cdot s=(2\cdot s-2\cdot(-s))+2(1\cdot s+1\cdot(-s))\in H[S].\]
Because $q$ is a linear combination of elements of $S$, we have $4q\in H[S]$.
Hence $\pi_1(X_0)$ is torsion and \cref{thm:torsion-implies-chi4} implies $\chi(\Cay(G,S))=\chi(X_0)\geq 4$ if $X_0$ is not bipartite.
\end{proof}

\begin{proof}[Proof of \cref{cor:pi1-torsion}]
Suppose $\pi_1(\N(X))$ is torsion. To show that $\chi(X)\geq 4$, it suffices by \cref{thm:torsion-implies-chi4} to show that $\pi_1(X)$ is torsion.

Let $W$ be a closed walk in $X$. Then $W*W$ is a closed walk of even length in $X$. By \cref{prop:pi1even}, as well as the fact that $\pi_1(\mathcal{N}(X))$ is torsion, the $k$-fold concatenation $(W*W)^k$ is null-homotopic for some positive integer $k$. In other words, $W^{2k}$ is null-homotopic, and it follows that $\pi_1(X)$ is torsion.
\end{proof}

To motivate \cref{thm:homology}, it is helpful to keep in mind the following alternate perspective on the proof of \cref{thm:torsion-implies-chi4}. As noted in \cref{remark:pushforward}, $\wind(-,c)$ gives a group homomorphism $\pi_1(X)\to\Z$. The proof of \cref{thm:torsion-implies-chi4} shows that $\wind(W,c)\neq 0$ for any odd-length walk $W$. We derive a contradiction by observing that if $\pi_1(X)$ is torsion, then the only group homomorphism $\pi_1(X)\to\Z$ is the 0 map, as the image of $\pi_1(X)$ in $\Z$ must be torsion.

The proof of \cref{thm:homology} is a diagram chase that builds on this idea.

\begin{proof}[Proof of \cref{thm:homology}]
Let $X$ be a connected non-bipartite graph, and suppose that $H_1(\N(X))$ is torsion. Suppose for the sake of contradiction that $X$ admits a proper 3-coloring $c:V(X)\to\Z/3\Z$.

Recall from \cref{prop:pi1even} that $\pi_1(\N(X))\cong \pieven(X)$ is an index-2 subgroup of $\pi_1(X)$ and from \cref{remark:pushforward} that $\wind(-,c)$ gives a group homomorphism $\pi_1(X)\to\Z$. As discussed in the introduction, $H_1(\N(X))$ is the abelianization of $\pi_1(\N(X))$ \cite[Theorem 2A.1]{Hat}. By the universal property of the abelianization, any group homomorphism $\pi_1(\N(X))\to \Z$ factors through $H_1(\N(X))$. Thus, we have the following commutative diagram of group homomorphisms.

\begin{tikzcd}
\pi_1(\N(X))\cong\pieven(X)\arrow{rd}\arrow[hookrightarrow]{r}&\pi_1(X)\arrow{rr}{\wind(-,c)}&&\Z
\\&H_1(\N(X))\cong\mathrm{Ab}(\pi_1(\N(X)))\arrow[]{rru}{}&
\end{tikzcd}

Because $H_1(\N(X))$ is torsion, the image of any group homomorphism $H_1(\N(X))\to\Z$ must also be torsion; it follows that the homomorphism $H_1(\N(X))\to\Z$ must be the 0 map. Thus, the composition $\pieven(X)\to H_1(\N(X))\to\Z$ is the 0 map.

However, if $W$ is an odd-length closed walk in $X$ then the argument in the proof of \cref{thm:torsion-implies-chi4} implies that $\wind(W,c)$ is odd, and thus $\wind(W*W,c)=2\times\wind(W,c)\neq 0$. Because $\wind(W*W,c)$ is in the image of the composition $\pieven(X)\to\pi_1(X)\to\Z$, this composition cannot be the 0 map. Thus, we have a contradiction, and it follows that $X$ is not 3-colorable.
\end{proof}

\section{Further Remarks}\label{sec:concluding}

\subsection{Infinite groups and the axiom of choice} For infinite groups, our proof of \cref{thm:chi3iff} is fundamentally reliant on the axiom of choice. Notably, \cref{lem:3-chromatic} relies upon the existence of a transversal set for the cyclic subgroup $\langle x\rangle$, which in turn uses the axiom of choice---see \cite{Ker}. The transversal set allows us to translate the coloring on $\langle x\rangle$ to cosets of $\langle x\rangle$. Additionally, we elide the distinction between 2-colorability and avoiding odd cycles, but the equivalence between these properties for infinite graphs also requires the axiom of choice. It would be interesting to know whether one can avoid the use of the axiom of choice here. See, for example, \cite{Payne} for a discussion of abelian Cayley graphs with ``ambiguous'' chromatic number, i.e., graphs for which the chromatic number depends on the axiomatic framework.

For finitely generated abelian groups $G$, our proof does not require the axiom of choice. Indeed, if $G_0$ is a subgroup of $G$, a transversal set for $G_0$ may be computed algorithmically, and it follows without the axiom of choice that the chromatic number of a disconnected Cayley graph over $G$ equals the maximum chromatic number of a connected component.  In this case, we may even impose the additional requirement that the Cayley graph is connected.

\begin{theorem}\label{thm:connected-Cayley}
    Let $G$ be a finitely generated abelian group. There exists a symmetric subset $S$ of $G$ such that $\Cay(G,S)$ is connected and 3-chromatic if and only if $\exp(G)\notin\{1,2,4\}$.
\end{theorem}

\begin{proof}
If $\exp(G)\in\{1,2,4\}$ then the conclusion follows from \cref{thm:chi3iff}.

Suppose  $\exp(G)\notin\{1,2,4\}$. Because $G$ is finitely generated, it holds that $G=G_1\times\cdots\times G_m$, where each $G_i$ is a nontrivial cyclic subgroup of $G$ whose order is either infinite or a prime power. Moreover, at least one subgroup---without loss of generality, $G_1$---must have order not equal to $1$, $2$, or $4$. We notate $G$ and its subgroups using additive notation. For each $i$, let $x_i$ be a generator of $G_i$.

Set $S_1$ equal to $\{\pm x_1,\pm 2x_1\}$ if $|G_1|$ is infinite, $\{\pm x_1\}$ if $|G_1|$ is odd, and $\{\pm x_1, 2^{k-1}x_1\}$ if $|G_1|=2^k$. In all cases, $X_1:=\Cay(G_1,S_1)$ is a connected non-bipartite graph. If $|G_1|$ is infinite or odd then $\chi(X_1)=3$ as discussed in the proof of \cref{lem:3-chromatic}. If $|G_1|=2^k$ then $k\geq 3$ because $|G_1|\notin\{1,2,4\}$. In this case, we may 3-color $X_1$ by coloring $x_1,\ldots,2^{k-2}x_1$ with red and blue alternatingly, $(2^{k-2}+1)x_1,\ldots,2^{k-1}x_1$ with green and red alternatingly, and $(2^{k-1}+1)x_1,\ldots,2^kx_1$ with blue and green alternatingly, as pictured in \Cref{fig:Z16}.
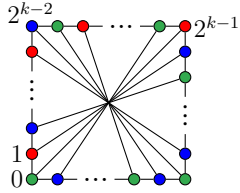
\begin{figure}[h]
\begin{tikzpicture}[rotate=45,scale=1.2,
	a/.style={vtx,draw=black,fill=red},
	b/.style={vtx,draw=black,fill=blue},
	c/.style={vtx,draw=black,fill=prettygreen}]

\path (0,0) coordinate[c] (d5)
	--++ (0.2,0.2) coordinate[a] (a1)
	--++ (0.2,0.2) coordinate[b]	 (a2)
	--++ (0.3,0.3) node[rotate=90] (a3) {\!...\!}
	--++ (0.3,0.3) coordinate[a] (a4)
	--++ (0.2,0.2) coordinate[b] (a5);
\path (a5)
	--++ (0.2,-0.2) coordinate[c] (b1)
	--++ (0.2,-0.2) coordinate[a] (b2)
	--++ (0.3,-0.3) node[rotate=0] (b3) {\!...\!}
	--++ (0.3,-0.3) coordinate[c] (b4)
	--++ (0.2,-0.2) coordinate[a] (b5);
\path (b5)
	--++ (-0.2,-0.2) coordinate[b] (c1)
	--++ (-0.2,-0.2) coordinate[c] (c2)
	--++ (-0.3,-0.3) node[rotate=90] (c3) {\!...\!}
	--++ (-0.3,-0.3) coordinate[b] (c4)
	--++ (-0.2,-0.2) coordinate[c] (c5);
\path (c5)
	--++ (-0.2,0.2) coordinate[b] (d1)
	--++ (-0.2,0.2) coordinate[c]	 (d2)
	--++ (-0.3,0.3) node[rotate=0] (d3) {\!...\!}
	--++ (-0.3,0.3) coordinate[b] (d4);
\draw (d5) -- (a1) -- (a2) -- (a3) -- (a4) -- (a5)
	-- (b1) -- (b2) -- (b3) -- (b4) -- (b5)
	-- (c1) -- (c2) -- (c3) -- (c4) -- (c5)
	-- (d1) -- (d2) -- (d3) -- (d4) -- (d5);
\foreach \i in {1,2,4,5} {
	\draw (a\i) -- (c\i) (b\i) -- (d\i);
}
\node[left,scale=0.8] at (d5) {$0$};
\node[left,scale=0.8] at (a1) {$1$};
\node[scale=0.8] at ($(a5)+(0.1,0.12)$) {$2^{k-2}$};
\node[right,scale=0.8] at (b5) {$2^{k-1}$};
\end{tikzpicture}

\caption{A 3-coloring of the graph $\Cay(\Z/2^k\Z,\{\pm 1,2^{m-1}\})$, which is isomorphic to $X_1$.}\label{fig:Z16}
\end{figure}

For $2\leq i\leq m$, set $S_i=\{\pm x_i\}$. Then $\Cay(G_i,S_i)$ is connected, with chromatic number $2$ if $|G_i|$ is infinite or even and chromatic number 3 if $|G_i|$ is odd. Setting $S=S_1\cup\cdots\cup S_m$, it holds that $\Cay(G,S)$ is isomorphic to the box product $\Cay(G_1,S_1)\squareop\cdots\squareop\Cay(G_m,S_m)$. (See, for example, \cite{Sab} for basic properties of box products of graphs.) As a box product of connected graphs, $\Cay(G,S)$ is connected. Moreover, $\chi(\Cay(G,S))$ equals the maximum of the chromatic numbers of the graphs $\Cay(G_{i},S_{i})$, which is $3$.
\end{proof}

\subsection{Generalizing \cref{thm:chi3iff}}

Our main theorem completely classifies all abelian groups that admit a $k$-chromatic Cayley graph for $k=3$. There is a similar classification for $k=2$: there exists a 2-chromatic Cayley graph over an abelian group $G$ if and only if $G$ contains an element of even or infinite order.  
\begin{question}
Is such a classification is possible for values of $k$ greater than $3$? 
\end{question}
 
 It would also be interesting to extend the classification in \cref{thm:chi3iff} to nonabelian groups $G$. It is not difficult to check that groups of exponent 1 or 2 are abelian; moreover \cref{lem:3-chromatic} implies that groups of exponent not dividing 4 admit 3-chromatic Cayley graphs. Thus, the only remaining case is Cayley graphs over nonabelian groups of exponent 4.
  
\begin{question}
Which nonabelian groups of exponent 4 admit a Cayley graph of chromatic number 3?
\end{question}

\subsection{Topological extensions}

Lov\'asz \cite{Lov} asked whether \cref{thm:lovasz} could be strengthened by considering homology groups instead of homotopy groups. It would be very interesting to find a strengthening of \cref{thm:lovasz}---using either homology or homotopy---that agrees with \cref{thm:homology} in the $k=1$ case.

\begin{problem}\label{problem:homology}
Suppose $X$ is a connected non-bipartite graph with $H_1(X)$ torsion. For $k\geq 2$, provide conditions on $H_2(\N(X)),\ldots,H_k(\N(X))$ (or on $\pi_2(\N(X)),\ldots,\pi_k(\N(X))$) which imply that $\chi(X)\geq k+3$.
\end{problem}

\cref{problem:homology} is additionally motivated by the study of chromatic numbers of binary Cayley graphs, i.e., Cayley graphs over the group $G=(\Z/2\Z)^m$. In 1992, Payan \cite{Pay} noted that many authors thought the chromatic number of binary Cayley graphs was always a power of 2. In that same article, he disproved this conjecture by constructing a binary Cayley graph with chromatic number $7$. Several other counterexamples \cite{Kok,Lau} to this conjecture have been given with chromatic number 13 or 14---these results are computational in nature. Conversely, in a 2014 online posting \cite{Roy}, Royle writes, ``I simply cannot find any cubelike graph with chromatic number 5, and according to Brouwer's website, this is unknown.''  To the authors' knowledge, it remains an open question as to whether a binary Cayley graph with chromatic number $5$ exists.

It seems plausible that an answer to \cref{problem:homology} would allow us to prove larger lower bounds on the chromatic numbers of binary Cayley graphs. In contrast, one would expect a binary Cayley graph with large chromatic number to have a large discrete fundamental group, and thus not be amenable to a direct application of \cref{thm:lovasz}.

\subsection{Limitations of topological techniques}\label{sub:failure}
We conclude with an example that highlights the limitations of our topological techniques to lower-bound chromatic numbers. We show that the converse of \cref{thm:torsion-implies-chi4} does not hold, and that in fact $\chi(X)$ is not always determined by $\pi_1(X)$. 

Fix $n\geq 9$ and let $X_n=\Cay(\Z/n\Z,\{\pm 1,\pm 2\})$; this graph is connected and non-bipartite. Moreover, in any proper 3-coloring of $X_n$, any two vertices $x$ and $x+3$ must receive the same color, because the triples $\{x,x+1,x+2\}$ and $\{x+1,x+2,x+3\}$ both induce subgraphs of $X_n$ isomorphic to $K_3$. If $n$ is not a multiple of 3, it follows that every vertex in $X_n$ must receive the same color, implying that $X_n$ is not 3-colorable.\footnote{In fact, $\chi(X_n)=4$, as one can show directly without too much trouble.  This is a special case of a general formula due to Heuberger \cite{Heu} for the chromatic number of a circulant graph with two generators.} However if $n$ is a multiple of 3 then $\chi(X_n)=3$, with the 3-coloring given by reducing modulo 3.

Applying \cref{prop:pi1-for-abelian-Cayley-graph} to compute $\pi_1(X_n)$ with $S=\{\pm 1,\pm 2\}$, one observes that
\[R[S]=\{a\cdot 1+b\cdot(-1)+c\cdot2+d\cdot-2\mid a,b,c,d\in\Z, a-b+2c-2d\equiv 0\ (\mathrm{mod}\ n)\}\]
and $H[S]=\{a\cdot 1+a\cdot(-1)+c\cdot 2+c\cdot(-2)\mid a,c\in\Z\}$. It follows that $\pi_1(X_n)\cong R[S]/H[S]\cong\Z\times\Z$, generated by $2\cdot 1-1\cdot 2$ and $k\cdot 1$. Thus, the discrete fundamental group cannot always be used to distinguish between 3-chromatic and 4-chromatic graphs $X_n$. It would be interesting to find a topological parameter that distinguishes the chromatic number in this example.

\subsection{Quadrangulations of surfaces}\label{sub:failure}

We remark that the results of this paper can be used to obtain information about the chromatic number of a quadrangulation $Q$ of a surface $Y$.  For we have that $\pi_1(Q)\cong\pi_1(Y)$.  Here $\pi_1(Q)$ refers to the discrete fundamental group discussed in this article, whereas $\pi_1(Y)$ is the ordinary topological fundamental group.  To see that $\pi_1(Q)\cong\pi_1(Y)$, note that (as discussed in \cite{matsushita} and \cite{wrochna}) for a graph $X$ we have that $\pi_1(X)$ is the fundamental group of the CW complex constructed by regarding vertices as $0$-cells, edges as $1$-cells, and $4$-cycles as $2$-cells, with the natural attaching maps.

In \cite{You}, Youngs proves that no quadrangulation of the real projective plane $\mathbb{R}P^2$ is $3$-chromatic.  Youngs' theorem follows immediately from \cref{cor:pi1-torsion}, because $\pi_1(\mathbb{R}P^2)\cong\mathbb{Z}/2\mathbb{Z}$.

We also find that no quadrangulation of a compact orientable surface $Y$ of genus $\geq 2$ can be an abelian Cayley graph.  For in this case $\pi_1(Y)$ is nonabelian, but \cref{prop:pi1-for-abelian-Cayley-graph} shows that the discrete fundamental group of an abelian Cayley graph is abelian.

\section*{Acknowledgments}

We thank the reviewers for many helpful comments.  We also thank Mark Ellingham and Florian Frick for independently pointing out that Youngs' theorem follows as a consequence of \cref{cor:pi1-torsion}.


\end{document}